\documentclass[11pt,svgnames]{article}

\usepackage{geometry}
\usepackage{graphicx}

\usepackage{amsthm}
\usepackage{amsmath}
\usepackage{amsfonts}
\usepackage{amssymb}

\newtheorem{theorem}{Theorem}
\newtheorem{lemma}[theorem]{Lemma}
\newtheorem{observation}[theorem]{Observation}
\newtheorem{proposition}[theorem]{Proposition}
\newtheorem{claim}[theorem]{Claim}
\newtheorem{corollary}[theorem]{Corollary}
\newtheorem{conjecture}[theorem]{Conjecture}

\usepackage{parskip}

\newcommand{\tlir}{\operatorname{tlir}}
\newenvironment{claimproof}[1]{%
  \begin{proof}[Proof of Claim~\ref{#1}]%
}{\end{proof}}

\title{Locally Irregular Total Colorings of Graphs}
\author{Anna Flaszczyńska\thanks{Department of Discrete Mathematics, AGH University of Krakow, Poland}, Aleksandra Gorzkowska\footnotemark[1], Igor Grzelec\footnotemark[1] \thanks{The corresponding author. Email:  grzelec@agh.edu.pl}, \\ Alfréd Onderko\thanks{Institute of mathematics, P.J. Šafárik University, Košice, Slovakia}, Mariusz Woźniak\footnotemark[1]}

\begin{document}
\maketitle

\begin{abstract}
A \textit{total graph} is an ordered triple $(V_0, V_1, E)$, where $V_0, V_1$ are the sets of empty and full vertices, respectively, $V_0 \cap V_1 = \emptyset$, and the set of edges $E$ is a subset of  \(\binom{V_0 \cup V_1}{2}\)  $(E\cap(V_0 \cup V_1)=\emptyset)$. A simple graph is a total graph in which all vertices are full. We say that a total graph $G$ is \textit{locally irregular} if every two adjacent vertices have different total degrees, where by the \textit{total degree} of a vertex $v$ in $G$ we mean the number of edges in $G$ that contain $v$ plus 1 if $v$ is full, or plus 0 if $v$ is empty. A total coloring of a graph $G$ whose colors induce locally irregular total subgraphs is called \textit{locally irregular total coloring}, and the minimum number of colors required in such a coloring of $G$ is denoted by $\tlir(G)$. In 2015, Baudon, Bensmail, Przybyło, and Woźniak conjectured that $\tlir(G)\leq 2$ for every graph $G$. In this paper, we prove this conjecture for cacti, subcubic graphs, and split graphs. We also provide a general upper bound for $\tlir(G)$ depending on the chromatic number of $G$, and a constant upper bound if $G$ is planar or outerplanar. In our proofs, we utilize special decompositions of graphs and the connection between acyclic vertex coloring and locally irregular total coloring.\\

\end{abstract}

\section{Introduction}\label{introduction}

In this paper, we consider simple finite graphs and multigraphs. 
Let $G=(V, E)$ be a graph. We call a function $f: E \rightarrow \{1, 2, \dots, k\}$ an \textit{edge coloring} of $G$. We put $\sigma (v)=\sum \limits_{v \in e}f(e)$ for each vertex $v$ of $G$. We say that two vertices $u$ and $v$ are \textit{distinguished} if $\sigma (u)\neq \sigma (v)$. Such an edge coloring of $G$ can be interpreted as creating a multigraph from $G$ in which we replace each edge by $f(e)$ parallel edges. In that case, $\sigma (v)$ is the degree of the vertex $v$ in the multigraph $G'$ obtained from the graph $G$. Such a multigraph is \textit{locally irregular} if adjacent vertices have different degrees. The function $\sigma (v)$ creates a proper vertex coloring of $G$ if all neighboring vertices are distinguished in the edge coloring of $G$. This concept was first developed in 2004 by Karoński, Łuczak, and Thomason in~\cite{Karonski2004}. They introduced a parameter $\chi_{\Sigma}(G)$, which is the smallest $k$ such that in an edge coloring of $G$ all adjacent vertices are distinguished. Such a coloring is called \textit{neighbor-sum-distinguishing}. They also stated the following well-known conjecture in~\cite{Karonski2004}.  

\begin{conjecture}[1-2-3 Conjecture]
For every graph $G$ without isolated edges, $\chi_{\Sigma}(G)\leq 3$.
\end{conjecture}

Before its resolution, the 1-2-3 Conjecture garnered significant attention due to its elegance. At the beginning, for every graph $G$ without isolated edges, Addario-Berry, Dalal, McDiarmid, Reed, and Thomason proved that  $\chi_{\Sigma}(G)\leq 30$~\cite{Addario-Berry2007}. This upper bound was subsequently reduced to 16 by Addario-Berry, Dalal, and Reed~\cite{Addario-Berry2008} and to 13 by Wang and Yu~\cite{Wang2008}. Then, for many years, the best known result about this conjecture was that every graph without isolated edges admits $\chi_{\Sigma}(G)\leq 5$ and was shown by Kalkowski, Karoński, and Pfender in~\cite{Kalkowski2010}. In the last few years, the conjecture has been intensively investigated and in case of regular graphs Przybyło proved in~\cite{Przybylo2021} that every $d$-regular graph $G$, where $d\geq 2$, admits $\chi_{\Sigma}(G)\leq 4$ and if $d\geq 10^8$, then $G$ admits $\chi_{\Sigma}(G)\leq 3$. Next, Keusch proved that every graph without isolated edges admits $\chi_{\Sigma}(G)\leq 4$ in~\cite{Keusch2023}. In 2024, the 1-2-3 Conjecture was confirmed by Keusch in~\cite{Keusch2023}. Additionally, regarding computational complexity, Dudek and Wajc proved that determining whether a given graph $G$ admits  $\chi_{\Sigma}(G)= 2$ is NP-complete~\cite{Dudek2011}, while Thomassen, Wu, and Zhang showed that this problem is polynomial-time solvable for bipartite graphs~\cite{Thomassen2016}. 

The above-mentioned problem was an inspiration for Przybyło and Woźniak to state the total version of that problem in~\cite{Przybyło2010}. We call a function $g: V \cup E \rightarrow \{1, 2, \dots, k\}$ a \textit{total coloring} of $G$. We put $\sigma^t (v):=g(v)+\sum \limits_{v \in e}g(e)$ for each vertex $v$ of $G$. We say that two neighboring vertices $u$ and $v$ are \textit{distinguished} if $\sigma^t (u)\neq \sigma^t (v)$. The corresponding parameter $\chi^t_{\Sigma}(G)$ is the smallest $k$ such that in a total coloring of $G$ all adjacent vertices are distinguished. We call this coloring \textit{total neighbor-sum-distinguishing}. In 2010, Przybyło and Woźniak stated the following conjecture in~\cite{Przybyło2010}.

\begin{conjecture}[1-2 Conjecture]
For every graph $G$ we have $\chi^t_{\Sigma}(G)\leq 2$.
\end{conjecture}

Initially it was proved that $\chi^t_{\Sigma}(G)\leq 11$ for every graph~$G$~\cite{Przybyło2010}. Since a graph $G$ satisfies $\chi^t_{\Sigma}(G)= 1$ if and only if every pair of adjacent vertices has different degrees, Przybyło noted that regular graphs are the most difficult in~\cite{Przybylo2008}. He also improved the upper bound to $\chi^t_{\Sigma}(G)\leq 7$ for arbitrary regular graph~$G$~\cite{Przybylo2008}. Later, Kalkowski proved that every graph has a total neighbor-sum-distinguishing coloring where vertices are colored using colors from the set $\{1, 2\}$ and edges are colored using colors from the set $\{1, 2, 3\}$ in~\cite{Kalkowski2009}. Recently, the 1-2 Conjecture was proved for every regular graph by Deng and Qiu in~\cite{Deng2025}. Finally, in 2025, the 1-2 Conjecture was proved for an arbitrary graph by Deng and Qiu in~\cite{Deng202512}. Although the 1-2 Conjecture may be less widely known than the 1-2-3 Conjecture, it has played a crucial role in advancing related research. More precisely, among others, Kalkowski’s algorithmic approach to the 1-2 Conjecture inspired Kalkowski, Karoński, and Pfender to reduce the general constant upper bound for the number of colors in neighbor-sum-distinguishing edge coloring to 5.

Neighbor-sum-distinguishing edge coloring and total neighbor-sum-distinguishing coloring were an inspiration for Baudon, Bensmail, Przybyło, and Woźniak to state two new problems with a bit different approach to the problem of local irregularity of graphs in \cite{Baudon2015}. We start with recalling the notion of locally irregular graphs and presenting some results on locally irregular colorings that are, in fact, decompositions into locally irregular subgraphs.

We call a graph $G$ \textit{locally irregular} if every two adjacent vertices of $G$ have distinct degrees. By a  \textit{locally irregular coloring} of a graph $G$, we mean an edge coloring of $G$ in which color classes induce locally irregular subgraphs. The corresponding parameter $\mathrm{lir}(G)$, known as the \textit{locally irregular chromatic index}, is the minimum number of colors required in a locally irregular coloring of $G$ if such a coloring exists. 

We can easily check that odd-length paths and cycles do not have a locally irregular coloring. Baudon, Bensmail, Przybyło, and Woźniak proved that only odd-length paths, odd cycles, and a special family of cacti (i.e., connected graphs with edge-disjoint cycles) do not have locally irregular coloring (are not locally irregular colorable) in~\cite{Baudon2015}. They also stated the conjecture that for every locally irregular colorable graph $G$ we have $\mathrm{lir}(G)\leq 3$. 

Since all locally irregular uncolorable graphs are special cacti, Sedlar and Škrekovski investigate $\mathrm{lir}(G)$ for every colorable cactus $G$ in~\cite{Sedlar2024} and~\cite{Sedlar2021}. During their research~\cite{Sedlar2021}, a single cactus $B$, see~Figure~\ref{graph_B}, satisfying $\mathrm{lir}(B) = 4$ was found. The bow-tie graph $B$ remains the only known colorable graph $G$ that has $\mathrm{lir}(G)=4$. Due to this finding, Sedlar and Škrekovski improved the Local Irregularity Conjecture:
\begin{conjecture}[Local Irregularity Conjecture~\cite{Baudon2015}, \cite{Sedlar2024}]
\label{graph3}
Every connected graph $G \neq B$ that is locally irregular colorable satisfies $\mathrm{lir}(G)\leq 3$.
\end{conjecture}

\begin{figure}[h!]
    \centering
    \includegraphics[width=4cm]{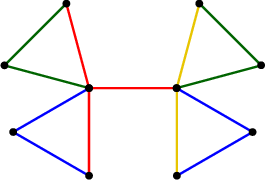}
    \caption{Locally irregular coloring of the bow-tie graph $B$ using four colors.}
    \label{graph_B}
\end{figure}

The Local Irregularity Conjecture was confirmed for some graph classes. Among others, trees~\cite{Baudon2015}, complete graphs~\cite{Baudon2015}, graphs with minimum degree at least $10^{10}$~\cite{Przybylo2016}, split graphs~\cite{Lintzmayer2021} and $r$-regular graphs, where $r\geq10^7$~\cite{Baudon2015}. For a planar graph $G$, it was proven that $\operatorname{lir}(G)\leq 15$ in~\cite{Bensmail2020} and for a subcubic graph $G$ that $\operatorname{lir}(G)\leq 4$ in~\cite{Luzar2018} (in both cases we assume that $G$ is locally irregular colorable). For every locally irregular colorable graph $G$, Bensmail, Merker, and Thomassen~\cite{Bensmail2017} proved that $\mathrm{lir}(G)\leq 328$, which was decreased to 220 by Lu\v zar, Przybyło, and Soták~\cite{Luzar2018}.

In this paper, we focus on the total version of the above-mentioned problem introduced in~\cite{Baudon2015}. We start with some definitions. A \textit{total graph} is an ordered
triple $(V_0, V_1, E)$, where $V_0$, $V_1$ are the sets of empty and full vertices, respectively, $V_0 \cap V_1 = \emptyset$, and the set of edges $E$ is a subset of  \(\binom{V_0 \cup V_1}{2}\)  $(E\cap(V_0 \cup V_1)=\emptyset)$. Note that a normal graph is a total graph in which all vertices are full. We call a total graph $G=(V_0, V_1, E)$ \textit{locally irregular} if for every edge $uv \in E$, the total degree of $u$ is distinct from the total degree of $v$, where by the \textit{total degree} of a vertex $v$ we mean the number of edges in $E$ containing $v$ plus 1 if $v \in V_1$ (or plus 0 if $v \in V_0$). By a \textit{locally irregular total coloring} of a graph $G$, we mean a total coloring of $G$ in which color classes induce locally irregular total subgraphs. The corresponding invariant $\mathrm{tlir}(G)$, called the \textit{locally irregular total chromatic index}, is the minimum number of colors required in a locally irregular total coloring of $G$. To keep our notation short, we will use TLIR coloring instead of locally irregular total coloring and TLIR $k$-coloring instead of locally irregular total coloring with $k$ colors.

Every locally irregular $k$-coloring of a graph $G$ can be extended to a TLIR $k$-coloring of $G$ by assigning the same color to all the vertices.
This observation provides us with $\tlir(G) \leq 220$ for every locally irregular colorable graph $G$ from Lu\v zar, Przybyło, and Soták's result~\cite{Luzar2018}. 
One of the results presented in this paper is that every cactus graph $G$ has $\tlir(G) \leq 2$.
Since all connected graphs that are not locally irregular colorable are special subcubic cacti, we get the general constant upper bound $\tlir(G) \leq 220$ for every graph $G$.

In~\cite{Baudon2015}, Baudon, Bensmail, Przybyło, and Woźniak formulated the following conjecture:    
\begin{conjecture}\label{main_conjecture}
    For every graph $G$, we have $\tlir(G) \leq 2$.
\end{conjecture}
Initially, this conjecture was proved for bipartite graphs.
To better utilize TLIR colorings of bipartite graphs, we reformulate this result in the following way:
\begin{proposition}\label{prop_bip}
    If $G=(X, Y, E)$ is a bipartite graph, then it admits a TLIR red-blue coloring such that every edge is red, every $x \in X$ has even total red degree, and every $y \in Y$ has odd total red degree.
\end{proposition}
\begin{proof}
    Color every edge of $G$ red. If, for some $x \in X$, $\deg_G(x)$ is even, color it blue; otherwise, color it red. 
    Similarly, $y \in Y$ is colored blue if $\deg_G(y)$ is odd, and it is colored red otherwise.
    Total red degrees are even for vertices in $X$, and odd for vertices in $Y$. 
    The blue subgraph is edgeless, so it is also locally irregular.
\end{proof}

Results of Deng and Qiu~\cite{Deng202512, Deng2025} provide us with a total neighbor-sum-distinguishing coloring of every graph $G$ using weights 0 and 1 on edges and vertices of $G$.
These weights may represent colors of edges, e.g., elements of $G$ with weight 0 are blue, and those with weight 1 are red.
Since neighbors are distinguished by sums, they have different degrees in the red subgraph of $G$. From regularity, we get that they also have different degrees in the blue subgraph, since the sum of the red and blue degrees of every vertex is the constant $k+1$.
We get
\begin{proposition}[Corollary of Theorem 1 in~\cite{Deng2025}]\label{prop_reg}
    If $G$ is a regular graph, then $\tlir(G) \leq 2$.
\end{proposition}

In this paper, we prove that Conjecture~\ref{main_conjecture} holds for cactus graphs, subcubic graphs, and split graphs (see Theorems~\ref{thm_cactus}, \ref{thm_subcubic}, and~\ref{thm_split}).

In Section~\ref{sec:chromatic}, we provide an upper bound on $\tlir(G)$ that $\tlir(G) \leq 2\chi(G) - 2$ if the chromatic number $\chi(G)$ is at least 2 (Theorem~\ref{thm_k_chromatic}).
Note that, for graphs $G$ with small chromatic number, in particular $\chi(G) \leq 110$, the bound from Theorem~\ref{thm_k_chromatic} is better than the general result $\tlir(G) \leq 220$ derived from the result of Lu\v zar, Przybyło, and Soták~\cite{Luzar2018}.  

Finally, in Section~\ref{sec:acyclic_star_tlir} we describe a connection between acyclic vertex colorings, star colorings, and TLIR colorings.
Using this, we prove that $\tlir(G) \leq 3$ if $G$ is outerplanar, and $\tlir(G) \leq 5$ if $G$ is planar (Theorems~\ref{thm_planar} and~\ref{thm_outerplanar}).

\section{Proof of Conjecture~\ref{main_conjecture} for cacti, subcubic, and split graphs}
    First, we prove a lemma that is useful for proving that Conjecture~\ref{main_conjecture} holds for cacti and subcubic graphs.

    \begin{lemma}\label{lemma_pendant_tree}
        Let $v$ be a vertex of a graph $G$ with $\deg_G(v) \leq 2$. Let $T$ be a tree such that $V(T) \cap V(G) = \{v\}$. 
        If $\tlir(G) \leq 2$, then there is a red-blue TLIR coloring of $G \cup T$ where all edges of $T$ have the same color, and all neighbors of $v$ in $T$ have the same color. 
    \end{lemma}
    \begin{proof}
        Consider a TLIR red-blue coloring $c$ of $G$. 
        
        First, suppose that there is at most one color used on edges incident to $v$ in $c$ ($v$ may be an isolated vertex in $G$).
        Let this color be blue. 
        Take a coloring $c'$ of $T$ where all edges are red, and $v$ has color $c(v)$, see Proposition~\ref{prop_bip}.
        Colorings $c$ and $c'$ form a TLIR 2-coloring of $G \cup T$ where all edges of $T$ have the same color.

        Now, suppose that both colors are used on edges incident to $v$.
        Hence $\deg_G(v) = 2$.
        Let $c(u_1v)$ and $c(u_2v)$ be red and blue, respectively. 
        Uncolor $v$ to obtain a partial coloring $c_0$ of $G \cup T$, where $T$ is not colored (neither vertices nor edges).

        Consider all four possible red-blue TLIR colorings of $T$ from Proposition~\ref{prop_bip}: all edges red and $v$ blue or red, or all edges blue and $v$ blue or red.
        Observe these colorings of $T$ produce different red and blue degrees of $v$: $0$ and $\deg_T(v)+1$, $1$ and $\deg_T(v)$, $\deg_T(v)$ and $1$, or $\deg_T(v) + 1$ and $0$.
        At least one of these colorings does not produce a conflict on $u_1v$ and $u_2v$ when combined with $c_0$.

        However, in such a coloring of $G \cup T$, there is one more red and blue edge incident to $v$, which could create a conflict between $v$ and its neighbors in $T$. If this is the case, switch the colors of all vertices in $T$ except $v$.
    \end{proof}

\subsection{Cacti}

    Let $C$ be a cycle of a cactus $G$.
    If $G - E(C)$ contains at most one component with cycles, then we call $C$ pendant.
    On the other hand, a tree $T$ in $G$ is called pendant if $G - V(T)$ has one connected component.

    \begin{observation}
        Every cactus with at least one cycle contains a pendant cycle.
    \end{observation}
    \begin{observation}\label{cactus_good_x}
        Every cactus $G$ contains a vertex $x$ such that at least $\deg_G(x)-2$ edges incident to $x$ lie on pendant cycles or pendant trees.
    \end{observation}
    It is possible to formally prove these two observations using a block-cut tree $BC_G$ of $G$, i.e., a graph whose vertices are the blocks $B$ and the cut-vertices $A$ of $G$, and where $ab \in E(BC_G)$ if and only if $a$ is a cut-vertex contained in the block $b$.
    Blocks of a cactus are cycles and bridges; we refer to the corresponding vertices in $BC_G$ as cycle-vertices and bridge-vertices. 
    The vertices of $BC_G$ that correspond to cut-vertices of $G$ will simply be called cut-vertices.
    We reduce $BC_G$ to a smaller tree $BC_G'$ by repeatedly applying the following step: if the tree contains a leaf which is a bridge-vertex or a cut-vertex, remove it.
    The leaves of the resulting tree $BC_G'$ correspond to pendant cycles.
    Moreover, the second vertex on a longest path in $BC_G'$ is a cut-vertex $x$ described in Observation~\ref{cactus_good_x}. 

    Such a structure of pendant cycles was used in the study of locally irregular colorings before; see, for example, \textit{grapes} in~\cite{Sedlar2024} (equivalent to Observation~\ref{cactus_good_x}).
    With these structural properties in mind, we can prove the main theorem of this section:
    \begin{theorem}\label{thm_cactus}
        If $G$ is a cactus graph, then $\tlir(G) \leq 2$.
    \end{theorem}
    \begin{proof}
    Suppose that the theorem is false.
    Let $G$ be the minimal (with respect to $|E(G)|$) cactus with $\tlir(G) \geq 3$.

    Note that $G$ contains cycles; 
    otherwise $G$ is a tree, and by Proposition~\ref{prop_bip} $\tlir(G) \leq 2$.

    Let $C$ be a pendant cycle of $G$, and let $x$ be a vertex of $C$ such that every vertex $x' \in V(C) \setminus \{x\}$ lies in a tree component of $G - E(C)$.
    We claim that every such $x' \in V(C) \setminus \{x\}$ is a 2-vertex.
    If this was not the case, we could split $G$ into a smaller cactus with $\deg_{G'}(x') = 2$ and a tree $T$ such that $V(G') \cap T = \{x'\}$.
    After this, we could apply Lemma~\ref{lemma_pendant_tree}.

    Therefore, every pendant cycle contains at most one vertex of degree at least 3.
    Take a vertex $x$ described in Observation~\ref{cactus_good_x}.
    Let $C_1, \dots, C_k$ and $T_1, \dots, T_\ell$ be pendant cycles and trees attached to $x$ such that at most two edges $xy_1, xy_2$ are not contained in any of them.
    We transform $G$ into a cactus $G'$ by opening the cycles $C_1, \dots, C_k$:
    If $C_i = (x,x_1, \dots, x_n)$, we remove vertices $x_2, \dots, x_{n-1}$ or an edge $x_1x_n$ if $C_i$ has lenght 3.

    After this, $G'$ can be split into a tree $T_x$ and a smaller cactus $G'_x$ where $V(T_x) \cap V(G_x') = \{x\}$, and $\deg_{G_x'}(x) \leq 2$.
    We use Lemma~\ref{lemma_pendant_tree} to obtain a locally irregular total coloring $c'$ of $G'$ where all edges of $T_x$ are blue, and all neighbors of $x$ in $T_x$ have the same color.

    Now, we extend $c'$ into a locally irregular total coloring of $G$. The coloring is independently extended for each $C_i$. 
    We distinguish several cases, depending on the length of $C_i$ and the colors of $x_1$ and $x_n$. 
    Note that, if the length of $C_i$ is 3 or 5 (Cases 1 and 2), some recoloring of the vertices $x,x_1,x_2$ and the edge $xx_2$ is needed.

    \textbf{Case 1.} $C_i$ has length 3.
    If the blue degree of $x$ is different from 2, then color $x_1$ blue, color $x_1x_2$ and $x_2$ red. 
    
    If the blue degree of $x$ equals 2, then all edges incident to $x$ except $xx_1$ and $xx_2$ are red. Recolor $x$ to blue and $xx_2$ to red; this does not change the color degrees of $x$. Next, color $x_1$ and $x_1x_2$ blue. Coloring the vertex $x_2$ red or blue does not create a conflict on $x_1x_2$ since the blue degree of $x_1$ is 3. Hence, we can color $x_2$ with a color that does not create a conflict on the red edge $xx_2$.

    \textbf{Case 2.} $C_i$ has length 5.
    If the blue degree of $x$ equals 2, then color $x_1$ and $x_1x_2$ blue, and all elements of path $(x_2, x_3, x_4)$ red. 
    
    If the blue degree of $x$ is not 2, then color $x_1$, $x_3$, and $x_4$ blue. 
    Next, color $x_2$ and all the edges of the path $(x_1,\dots, x_4)$ red.

    \textbf{Case 3.} $C_i$ has even length, i.e., $n$ is odd.
    Color the edges of the path $(x_1,x_2, \dots, x_n)$ red. 
    Next, color the vertices $x_3, x_5 \dots, x_{n-2}$ blue, and $x_2,x_4, \dots, x_{n-1}$ red.

    \textbf{Case 4.} $C_i$ has odd length, at least 7, i.e., $n \geq 6$ is even.
    If $x_1,x_n$ are blue in $c'$, then each edge of the path $(x_1,x_2, \dots, x_n)$ is colored red. 
    Next, the vertices $x_3, x_5, \dots, x_{n-1}$ are colored blue, and $x_2,x_4, \dots, x_{n-2}$ are colored red.

    If $x_1,x_n$ are red in $c'$, then color $x_1x_2$, $x_2$, $x_2x_3$ and $x_4$ red and $x_3$, $x_3 x_4$ blue. Then color the path $(x_4, \dots, x_n)$ as in Case 3.
\end{proof}

\subsection{Subcubic graphs}

\begin{theorem}\label{thm_subcubic}
    If $G$ is a subcubic graph, then $\tlir(G) \leq 2$.
\end{theorem}
\begin{proof}
Suppose that the theorem is false. 
Let $G$ be a subcubic graph with $\tlir(G) \geq 3$, the minimum value of 
\begin{align*}
    s(G) = \sum\limits_{v \in V(G)} 3 - \deg_G(v),   
\end{align*}
and the minimum number of vertices among the subcubic graphs $G'$ with $\tlir(G') \geq 3$ and $s(G') = s(G)$. 
From the minimality, we immediately get that $G$ is connected.
We prove three structural claims \linebreak on $G$.

The first one follows directly from Lemma~\ref{lemma_pendant_tree}.
\begin{claim}\label{claim_subcubic_1}
    $G$ does not contain a $1$-vertex. 
\end{claim}

\begin{claim}~\label{claim_subcubic_2}
    $G$ does not contain two adjacent $2$-vertices.
\end{claim}
\begin{claimproof}{claim_subcubic_2}
    Suppose to the contrary that there is a path $x_0,x_1,x_2,x_3$ in $G$ such that $x_1$ and $x_2$ are 2-vertices in $G$.
    Note that $x_1x_3$ might be an edge in $G$, or $x_1$ and $x_3$ might be the same vertex; we consider several cases.

    \textbf{Case 1.} $x_0 \neq x_3$ and $x_0x_3 \notin E(G)$.
    In this case, let $G'$ be the graph obtained from $G-x_1-x_2$ by adding the edge $x_0x_3$.
    We have $s(G') = s(G) - 2$.
    Take a red-blue TLIR coloring $c$ of $G'$.
    W.l.o.g. suppose that $x_0x_3$ is red, and the total red-degrees of $x_0$ and $x_3$ are $a$ and $b$, respectively.
    Moreover, assume that $a > b$; otherwise we switch the roles of $x_0$ and $x_3$.
    In total, we get six cases; $(a,b) \in \{(2,1),(3,1),(4,1),(3,2),(4,2),(4,3)\}$.
    To extend $c$ into a TLIR 2-coloring of $G$, color the edges $x_0x_1$ and $x_2x_3$ red and:
    \begin{itemize}
        \item[(i)] $x_1$ and $x_1x_2$ blue, and $x_2$ red, if $a \in \{2,3,4\}$ and $b \in \{1,3\}$;
        \item[(ii)] $x_1$ blue, $x_1x_2$ and $x_2$ red, if $a \in \{3,4\}$ and $b =2$. 
    \end{itemize}
    See Figure~\ref{subcubic2} for illustration.

    \begin{figure}[h!]
        \centering
        \includegraphics[width=13cm]{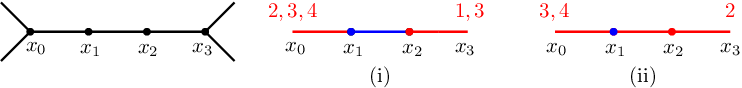}
        \caption{TLIR 2-colorings of $G$ from the case when $x_0 \neq x_3$ and $x_0x_3 \notin E(G)$.}
        \label{subcubic2}
    \end{figure} 

\textbf{Case 2.} $x_0 \neq x_3$, but $x_0x_3 \in E(G)$, i.e., vertices $x_0,x_1,x_2,x_3$ induce $C_4$.

If $\deg_G(x_0)=\deg_G(x_3) = 3$, then let $G' = G - x_1-x_2$.
Hence, $G'$ has the same number of 2-vertices, and we did not create any 1-vertex.
Thus $s(G') = s(G)$, but $G'$ is smaller.
Take a red-blue TLIR coloring of $G'$, and uncolor the vertices $x_0,x_3$ and the edge $x_0x_3$.
Depending on the colors of the edges incident to $x_0$ and $x_3$, $x_0y_0$ and $x_3y_3$, and the color degrees of $y_0$ and $y_3$ in this partial coloring, we extend this partial coloring into a TLIR 2-coloring of $G$. See Figure~\ref{subcubic3}.
Note also, that vertices $y_0$ and $y_3$ can be the same; the colorings presented in Figure~\ref{subcubic3} cover this case, too (cases (iii) and (viii) are, however, unobtainable).

\begin{figure}[h]
    \centering
    \includegraphics[width=11.5cm]{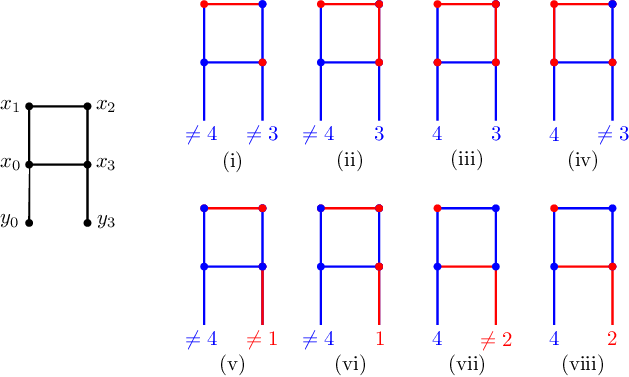}
    \caption{TLIR 2-colorings of $G$ from the case when $x_0 \neq x_3$, but $x_0x_3 \in E(G)$ in the first subcase.}
    \label{subcubic3}
\end{figure} 

Suppose that $\deg_G(x_0) = 3$ and $\deg_G(x_3) = 2$.
Then $C_4$ induced on $x_0,x_1,x_2,x_3$ is pendant in $G$.
Let $G' = G - x_1-x_2-x_3$.
Clearly, $s(G') < s(G)$.
Take a red-blue TLIR coloring $c$ of $G'$ where the edge incident to $x_0$ is blue.
We extend this coloring into a TLIR 2-coloring of $G$ using Proposition~\ref{prop_bip}: edges of the removed $C_4$ are red, and the color of $x_0$ in $c$ matches its color in the coloring of $C_4$.

If $\deg_G(x_0) = \deg_G(x_3) = 2$, then $G$ is a cycle on 4 vertices, which has a TLIR 2-coloring, see Proposition~\ref{prop_bip}.

\textbf{Case 3.} $x_0 = x_3$. 
If $\deg_G(x_0) = 2$, then $G$ is a cycle on three vertices. Regular graphs have TLIR 2-coloring, see Proposition~\ref{prop_reg}.

Therefore, assume that $\deg_G(x_0) = 3$.
Let $G' = G - x_1 - x_2$.
Clearly, $s(G') = s(G)$, but $G'$ is smaller.
Take a red-blue TLIR coloring of $G'$ where the edge incident to $x_0$, say $x_0y_0$, is blue.
Remove a color from $x_0$ and denote the obtained partial coloring by $c$.
In Figure~\ref{subcubic5} we provide an extension of $c$ into a TLIR 2-coloring of $G$, depending on the blue degree of $y_0$.

\begin{figure}[h!]
\centering
\includegraphics[width=6.5cm]{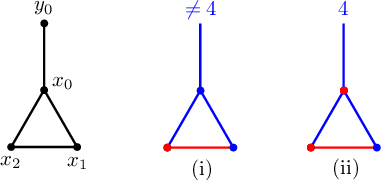}
\caption{TLIR 2-colorings of $G$ from the case when $x_0 = x_3$.}
\label{subcubic5}
\end{figure} 
\end{claimproof} 

\begin{claim}\label{claim_isolated_2-vertices}
    $G$ does not contain any $2$-vertices.
\end{claim}
\begin{claimproof}{claim_isolated_2-vertices}
    From Claims~\ref{claim_subcubic_1} and~\ref{claim_subcubic_2} we get that $G$ contains only 2-vertices and 3-vertices,   and 2-vertices are independent.
    Denote by $X$ the set of 2-vertices of $G$, and suppose $|X| \geq 1$.

    Denote by $W$ the graph which is obtained from a complete graph $K_4$ with vertices $w_1,w_2,w_3,w_4$ by subdividing the edge $w_1w_2$ with a vertex $w_0$.
    Let $G'$ be the graph obtained from $G$ by attaching a copy of $W$ by an edge $xw_0$ for every $x \in X$. See Figure \ref{subcubic6} for illustration.

\begin{figure}[h!]
\centering
\includegraphics[width=3.5cm]{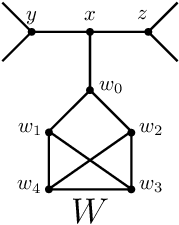}
\caption{A copy of the graph $W$ attached by an edge $xw_0$ to a vertex $x \in X$ from $G$.}
\label{subcubic6}
\end{figure} 

    Since $G'$ is cubic, it has a red-blue TLIR coloring, see Proposition~\ref{prop_reg}. 
    However, to prove Claim~\ref{claim_isolated_2-vertices} we will need the property of such a coloring, which is described in the proof of Theorem~1 from~\cite{Deng2025}.
    In this proof, the authors provide a construction of a neighbor-total-sum-distinguishing $\{0,1\}$-coloring of every graph $G$. 
    First, they take a maximum (in terms of cardinality) independent set $V_0$ in $G$, then the maximum independent set $V_1$ in $G-V_0$, etc. 
    In general $V_i$ is the maximum independent set of $G - \bigcup_{j=0}^{i-1}V_j$. 
    Then they provide a construction of a neighbor-total-sum-distinguishing $\{0,1\}$-coloring in which every vertex $v_i \in V_i$ has its total weight equal to $i$.
    Note that their construction does not specify which maximum independent set to take; we can choose any independent set with maximum cardinality as $V_0$.

    As we mentioned in Section~\ref{introduction}, this $\{0,1\}$-coloring, in the case of regular graphs, is a locally irregular total coloring; elements with weight 0 are red, and those with weight 1 are blue.
    In particular, this means that vertices from $V_0$ have total red-degree equal to 4, vertices from $V_1$ have total red-degree equal to 3, etc.

    In our case, take a maximum independent set $V_0$ such that $|V_0 \cap X|$ is minimal, i.e., $V_0$ contains as few original 2-vertices as possible.
    This means that if $y$ and $z$ are neighbors of some $x \in X$ in $G$, and $x \in V_0$, then both $y$ and $z$ have at least one other neighbor from $V_0$.
    Otherwise, we could exchange $x$ with $y$ or $z$.
    Moreover, if $x \in X$ is not in $V_0$, then from the copy of $W$ that is attached to $x$, exactly 2 vertices are in $V_0$. 
    We will assume that in this case $w_1,w_2 \in V_0$.

    Observe that at least one of $x,y,z,w_0$ is in $V_0$, because $V_0$ is a maximum independent set.
    We distinguish cases depending on whether both, one, or none of $x,y$ are in $V_0$.

    \textbf{Case 1.} $y,z \in V_0$. In this case, we simply remove $W$, and, in what is left, we have that the total red-degree of $x$ is at most 3, which is not in conflict with the total red-degrees of $y$ and $z$ that are equal to 4.

    \textbf{Case 2.} $y \in V_0$ and $z \notin V_0$ (the case when $y \notin V_0$ and $z \in V_0$ is symmetric).
    If there is a conflict after we remove $W$, it is between $x$ and $z$ (because the total red-degree of $y$ is 4 and the total red-degree of $x$ is at most 3). 
    Then we can change the color of $x$ to solve this conflict between $x$ and $z$, and such a change will not produce a conflict between $x$ and $y$.   

    \textbf{Case 3.} $x \in V_0$. 
    Since $V_0$ is a maximal independent set, both $y$ and $z$ have neighbors from $V_0$, i.e., there are at least two red edges incident to each of them. 
    Note that if we remove $W$ and produce a conflict between $x$ and one of its neighbors, we can also recolor $x$ to avoid this conflict. Such a recoloring of $x$ can, however, produce a conflict between $x$ and its other neighbor. 
    If this is the case, then the total red-degrees of $y$ and $z$ are 2 and 3.
    Suppose that the total red-degree of $y$ is 2. Then it is blue, and it is incident to one blue edge. 
    We can recolor $y$ to red, edge $xy$ to blue, and $x$ to red. 

    Observe that in all cases, we did not change the total color-degrees of neighbors of $x$. Using the fact that vertices from $X$ are independent, we can remove copies of $W$ one-by-one, and in each step solve the possible new conflicts between the 2-vertex $x$ and its neighbors.
    This yields a TLIR 2-coloring of $G$, a contradiction.
\end{claimproof}

Claims~\ref{claim_subcubic_1}, \ref{claim_subcubic_2}, and \ref{claim_isolated_2-vertices} yield that $G$ is, in fact, cubic.
However, $\tlir(G) \leq 2$ for every regular graph $G$, a contradiction.
\end{proof}

\subsection{Split graphs}

A \textit{split} graph is a graph $G$ with a clique $X$ such that $V(G) \setminus X$ is an independent set. 
We will denote by $G[X]$ the complete subgraph of $G$ induced on the vertices from $X$.

Note that $K_2$, $K_3$, and $P_4$ are the only split graphs that are not locally irregular colorable.
For other split graphs $G$, Lintzmayer, Mota, and Sambinelli~\cite{Lintzmayer2021} proved that $\operatorname{lir}(G) \leq 3$.
Moreover, they fully characterized those split graphs that require three colors in a locally irregular coloring.

\begin{lemma}[Corollary of Theorems 1.5 and 3.1 from~\cite{Lintzmayer2021}]\label{lemma_split}
    Let $G$ be a locally irregular colorable split graph with a clique $X = \{x_1, \dots, x_n\}$ and an independent set $Y = V(G) \setminus X$. 
    Let $d_i = |N(x_i) \cap Y|$.
    If $n \geq 3$, $d_1 \geq \cdots \geq d_n$, and
    $\operatorname{lir}(G) = 3$, then
    \begin{itemize}
        \item either $d_1 < \lfloor \tfrac{n}{2}\rfloor$ and $d_2 = 0$,
        \item or $d_1 = d_2 = 1$, $d_3 = 0$ and $n \in \{6,7,8\}$.
    \end{itemize}
\end{lemma}

Using the above lemma, Propositions~\ref{prop_reg} and~\ref{prop_bip}, and a simple observation that $\operatorname{lir}(G) \leq 2$ implies $\tlir(G) \leq 2$, we prove the following theorem. 

\begin{theorem}\label{thm_split}
    If $G$ is a split graph, then $\tlir(G) \leq 2$.
\end{theorem}
\begin{proof}
We keep the same notation as in Lemma~\ref{lemma_split}; 
$X = \{x_1, \dots, x_n\}$ be the maximal clique in $G$ such that $Y = V(G) \setminus X$ is an independent set, $d_i$ is the number of neighbors of $x_i$ in $Y$, and $d_1 \geq \cdots \geq d_n$. 

If $n \leq 2$, then $G$ is a tree, and the result follows from Proposition~\ref{prop_bip}.
If $Y = \emptyset$, i.e., $G$ is a complete graph, then the result follows from Proposition~\ref{prop_reg}.

in the following, we assume that $n \geq 3$ and $Y \neq \emptyset$.
By Lemma~\ref{lemma_split}, we need to consider only the following cases: 
\begin{itemize}
\item $d_1 < \lfloor \frac{n}{2} \rfloor$ and $d_2 = 0$, 
\item $d_1=d_2=1$, $d_3 = 0$ and $n\in \{6,7,8\}$.
\end{itemize}
In other cases, $G$ is either $P_4$ (use Proposition~\ref{prop_bip}), or $\operatorname{lir}(G) \leq 2$.

We start with the case when $d_1 < \lfloor \frac{n}{2} \rfloor$ and $d_2 = 0$. 
Proposition~\ref{prop_reg} yields a red-blue  TLIR coloring of $G[X]$. 
W.l.o.g., let $x_1$ have the largest total red-degree. 
The uncolored edges are pendant and incident to $x_1$; we color these edges red, and color the pendant vertices blue.
Since $x_1$ had the largest red-degree among the vertices of $X$, increasing its red-degree does not produce a conflict on $X$.
Vertices from $Y$ have red-degrees equal to 1, so there is no conflict between $x_1$ and its neighbors. 

\begin{figure}[h!]
\centering
\includegraphics[width=1\textwidth]{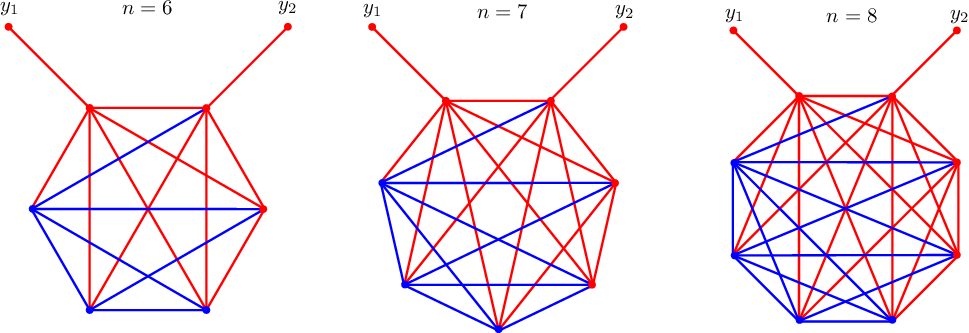}
\caption{TLIR 2-coloring of a split graph $G$ when $n\in \{6,7,8\}$ and $|Y| = 2$.}
\label{split_graph_fig}
\end{figure}

Next, we consider the case when $d_1=d_2=1$, $d_3 = 0$ and $n\in \{6,7,8\}$.
Vertices $x_1$ and $x_2$ can have a common neighbor from $Y$, or their neighbors might be distinct.
In Figure \ref{split_graph_fig}, we present a red-blue TLIR coloring of $G$ when the neighbors $y_1,y_2 \in Y$ of $x_1,x_2$ are distinct.
Keeping the same coloring, but identifying $y_1$ with $y_2$ yields a TLIR 2-coloring if $|Y| = 1$.
\end{proof}

\section{Bounds on \boldmath{$\tlir(G)$} using the chromatic number}\label{sec:chromatic}

The following lemma describes a possible TLIR partial 2-coloring of a bipartite graph with the properties that are crucial for its use in the proof that $\tlir(G) \leq 2\chi(G) - 2$. 

\begin{lemma}\label{lemma_partial_bipartite}
    Let $G$ be a bipartite graph with a bipartition $(X,Y)$. Suppose that $Y$ does not contain isolated vertices. Then there is a TLIR $2$-coloring of $G$ such that:
    \begin{itemize}
        \item[(i)] every vertex from $X$ is red or blue;
        \item[(ii)] every vertex from $X$ has even total red-degree;
        \item[(iii)] every vertex from $Y$ is uncolored; 
        \item[(iv)] every vertex from $Y$ has odd red-degree;
        \item[(v)] uncolored edges are a matching.
    \end{itemize}
    \end{lemma}
    
    \begin{proof}
        For each $y \in Y$ with $\deg_G(y) \equiv 0 \pmod{2}$, we take one of its neighbors $x \in X$; denote the graph induced on these edges by $G_B$.
        Let $G_R$ be the subgraph of $G$ induced on the remaining edges.

        Every component of $G_B$ is a star $S$.
        If $S$ has at least two edges, color these edges blue. 
        If $S$ is $K_2$, leave it uncolored.
        Color the edges of $G_R$ red.
        After this step, the uncolored edges of $G$ induce a matching, and every $y \in Y$ has odd total red-degree.
        Hence, (iv) and (v) hold.

        Now, we color the vertices from $X$.
        If $x \in X$ is incident to an even number of red edges, color it blue.
        If $x \in X$ is incident to an odd number of red edges, color it red.
        After this step, (i), (ii), and (iii) hold.

        Observe that the obtained partial coloring is TLIR.
        In the red subgraph, vertices from $X$ have even total degrees, and vertices from $Y$ have odd total degrees.
        In the blue subgraph, vertices in $Y$ have total degrees at most 1.
        However, if the total degree of some $x \in X$ is 1 in the blue subgraph, then $x$ is blue and it is not incident to any blue edges.
        This finishes the proof. 
    \end{proof}

    With repeated use of the previous lemma, we prove the main result of this section:
    \begin{theorem}\label{thm_k_chromatic}
        For every graph $G$ with $\chi(G) \geq 2$ we have $\tlir(G) \leq 2 \chi(G) - 2$.
    \end{theorem}
    \begin{proof}
        Let $k = \chi(G)$, and let $A_1, \dots, A_k$ be the color classes of a proper $k$-coloring of $G$.
        The result follows directly from Proposition~\ref{prop_bip} if $k = 2$.
        Hence, we may assume that $k \geq 3$ and the color classes $A_1, \dots, A_k$ are maximal in the sense that if $v \in A_i$, then $v$ has a neighbor in each $A_j$, $j < i$.

        First, we take a bipartite subgraph $B_1$ of $G$ induced on all edges between $A_1$ and $\bigcup_{i=2}^k A_i$.
        We color $B_1$ using a partial locally irregular total coloring described in Lemma~\ref{lemma_partial_bipartite} with colors 1 and 2.
        After this, every vertex from $A_1$ is colored, every vertex from $\bigcup_{i=2}^kA_i$ is uncolored, and some matching of $B_1$ is uncolored. 
        These uncolored edges will be colored in later steps.
        
        In a general step we will consider a bipartite graph $B_j$ induced on all edges between $A_j$ and $\bigcup_{i=j+1}^kA_i$ and all the uncolored edges incident to vertices from $A_j$ (that have their other ends in $\bigcup_{i=1}^{j-1}A_i$).
        In the end, when $B_{k-1}$ is considered, we include all uncolored edges incident to $A_{k-1}$ and $A_k$. 
        We color this last bipartite graph not using Lemma~\ref{lemma_partial_bipartite}, but in a slightly different way. See Figure \ref{fig2} for better illustration of the construction.

        \begin{figure}[h!]
            \centering
            \includegraphics[width=1\textwidth]{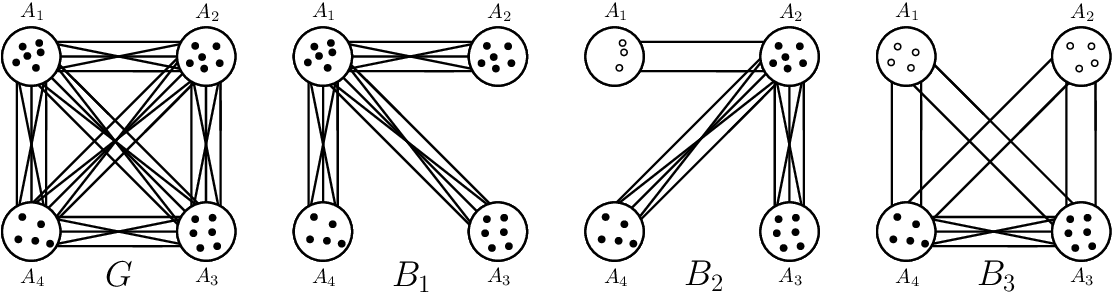}
            \caption{Graph $G$ with $\chi(G) = 4$ and bipartite subgraphs $B_1,B_2,B_3$ considered in the construction of TLIR 6-coloring of $G$. 
            Note that from Lemma \ref{lemma_partial_bipartite} we do not color vertices from color classes $A_2$, $A_3$, $A_4$ in $B_1$, and we do not color vertices from color classes $A_3$, $A_4$ in $B_2$. Additionally, in these partial colorings of $B_1$ and $B_2$, some independent edges might not be colored.}
            \label{fig2}
        \end{figure}

        Now, we will describe in detail how the partial coloring of $B_j$ looks. 
        Note that the created coloring of $G$ has the property that directly before $B_j$ is colored, every vertex from $\bigcup_{i=1}^{j-1}A_i$ is colored, and no vertex from $\bigcup_{i = j}^{k}A_i$ is colored.
       
        Consider now the bipartite graph $B_j$ for $1 \leq j \leq k-2$.
        Consider the partial coloring described in Lemma~\ref{lemma_partial_bipartite} which uses colors $2j-1$ and $2j$, where $X=A_j$ and $Y=B_j \setminus X$.
        Note that $Y$ contains all the vertices from $\bigcup_{i=1}^{j-1}A_i$ which are adjacent by uncolored edges to vertices from $A_j$.
        From (i) of Lemma~\ref{lemma_partial_bipartite} it follows that after this step, every vertex from $A_j$ is colored.
        Let $y \in V(B_j)$ be a vertex from $A_\ell$ for some $\ell \leq j-1$, and let $x \in X$ be its neighbor in $B_j$. 
        Since uncolored edges created in the step when $B_\ell$ was colored were independent, the degree of $y$ in $B_j$ is one.
        In the partial coloring of $B_j$, the total red-degree of $y$ is odd, see~(iv), and $y$ is uncolored. 
        Hence, the edge $xy$ is colored in the partial coloring of $B_j$. 
        Moreover, since in the partial coloring of $B_j$ the vertex $y$ is uncolored, in the coloring of $G$ we can preserve the color assigned to it during the coloring of $B_\ell$.

        After we colored every $B_j$ for $j \leq k-2$, the uncolored elements are: all edges between $A_{k-1}$ and $A_{k}$, some edges going from $A_k \cup A_{k-1}$ to $\bigcup_{i=1}^{k-2}A_i$, and all vertices in $A_{k-1}$ and $A_k$. 
        Recall that $B_{k-1}$ is the graph on all uncolored elements at this step.
        Let $U_{k-1}$ and $U_k$ be the sets of neighbors of $A_{k-1}$ and $A_{k}$ from $\bigcup_{i=1}^{k-2}A_i$ in $B_{k-1}$, i.e., for every vertex $u \in U_{k-1}$ there is a vertex $v$ from $A_{k-1}$ such that $uv$ is uncolored, and the similar holds for $U_{k}$ and $A_k$.
        
        Observe that $U_{k-1} \cap U_k = \emptyset$. 
        Suppose, to the contrary, that this is not the case, and there is some $u \in U_{k-1} \cap U_k$. 
        Let $x \in A_{k-1}$ and $y \in A_k$ be the neighbors of $u$ such that $ux$ and $uy$ would be uncolored. 
        Let $u \in A_\ell$. 
        However, when $B_\ell$ was colored, and some edges going from $A_\ell$ to $\bigcup_{i=\ell+1}^{k}A_i$ were left uncolored, these edges were independent; a contradiction.
        This also shows that the degree of every $u \in U_{k-1} \cup U_k$ is one in $B_{k-1}$.
        
        Let $(X,Y)$ be a bipartition of $B_{k-1}$ such that $X = A_{k-1} \cup U_k$ and $Y = A_{k} \cup U_{k-1}$.
        Consider a TLIR 2-coloring of $B_{k-1}$ with two colors $2k-3$ (red) and $2k-2$ (blue) such that every edge of $B_{k-1}$ is red, vertices from $X$ have even total red-degree, and vertices from $Y$ have odd total red-degrees (see Proposition~\ref{prop_bip}).
        To use this TLIR coloring of $B_{k-1}$ to complete the TLIR coloring of $G$, we need to change it to a partial coloring where vertices from $U_{k-1} \cup U_k$ are uncolored, because these vertices already have colors in the coloring of $G$.
        
        Let $u \in U_{k-1}$. Since $U_{k-1} \subseteq Y$, the total red-degree of $u$ is odd. Therefore, the only edge incident to $u$ is red, and $u$ itself is blue. In this case, we uncolor $u$ in $B_{k-1}$.
        On the other hand, if $u \in U_k$, then $u$ has even total red-degree. Hence, in this case, $u$ and an edge incident to it are red. 
        Uncolor $u$. If a conflict was created between $u$ and its neighbor $y$, it means that $y$ has total red-degree 1. Therefore, $y$ has degree 1 in $B_{k-1}$, because all edges are red. It is sufficient to recolor $y$ to red.
        
        After these adjustments, the partial total coloring is locally irregular, and every element of $B_{k-1}$, except the vertices of $U_{k-1} \cup U_k$, is colored.
        This finishes a TLIR $(2k-2)$-coloring of $G$.
    \end{proof}

\section{Acyclic vertex colorings, star colorings, and TLIR colorings}
\label{sec:acyclic_star_tlir}

Arboricity of a graph, i.e., the minimum number of spanning forests into which the graph can be decomposed, is a well-known and studied invariant.
A lot of special types of arboricity were studied, where some additional constraints are put on the forests (e.g., linear arboricity).
One of such concepts, star arboricity, asks for the minimum number of star forests into which a graph can be decomposed.
By a star forest, we mean a forest whose every component is a star, i.e., a tree in which at most one vertex has degree greater than 1.
By such a definition, $K_2$ is a star; we will say that $K_2$ is a trivial star.
Stars of larger size will be called nontrivial.

From the perspective of locally irregular decompositions, considering a star coloring of a graph seems to some degree natural: every nontrivial star is locally irregular.
However, since some colors may induce isolated edges (trivial stars), it is not true that the star coloring of a graph is always a locally irregular one (any star coloring of $P_4$, for example).
However, if we consider the problem of locally irregular total colorings, the use of colors on vertices gives hope that the star coloring could be somehow adjusted to yield a TLIR coloring.
On the other hand, even if a star coloring of a graph $G$ did not contain trivial stars, the number of required colors could be much larger than 2, which is the conjectured upper bound on $\tlir(G)$. 
For example, for a $d$-regular graph $G$, we have $\tlir(G) \leq 2$ (see Proposition~\ref{prop_reg}), but it is known that the minimum number of colors in a star coloring is at least $12d$ (see~\cite{Algor1989}).
Hence, the connection between star colorings of graphs and TLIR colorings is the most interesting for graphs with relatively small star arboricity.
Among such graphs, we have planar and outerplanar graphs, for which the star arboricity is 5 and 3, respectively~\cite{Hakimi1996}.

In~\cite{Hakimi1996}, Hakimi, Mitchem, and Schmeichel briefly described a connection between star colorings, acyclic (vertex) colorings, and $k$-degenerate graphs.
A proper vertex coloring is acyclic if every two color classes induce an acyclic subgraph (a forest).
A $k$-degenerate graph is a graph whose vertices can be ordered in a list $v_1, \dots, v_n$ such that every vertex $v_i$ has at most $k$ neighbors with smaller indices.
We, to some degree, describe the connection between $k$-degenerate graphs, acyclic vertex colorings, and star colorings here. 
In particular, we show how one can obtain an acyclic vertex 3-coloring of an outerplanar graph, how we can create a star coloring from such a vertex coloring, and that the combination of these two colorings is always a TLIR coloring.
Part of this can be proved in general, namely
\begin{theorem}\label{thm_acyclic_to_TLIR}
If $G$ has a vertex acyclic $k$-coloring, then $\tlir(G) \leq k$.
\end{theorem}
\begin{proof}
    Let $c_V$ be an acyclic vertex $k$-coloring.
    An edge coloring $c_E$ is constructed from $c_V$ in the following way: Let $a$ and $b$ be two colors in $c_V$. 
    Let $G_{a,b}$ be the forest induced on the vertices of color $a$ and $b$.
    In each component $H$ of $G_{a,b}$, choose any vertex as a root.
    Do a breadth-first search of $H$ from the root and color every edge outgoing from $u$ by the color $c_V(u)$, for every $u \in V(H)$. 
    Hence, every edge in $H$ has a color of the vertex that is closer to the root of $H$.
    We do this for every pair of colors.
    In the end, $c_E$ is an edge coloring of $G$.
    
    We claim that $c_E$ is in fact a star coloring, in which each nontrivial star of color $a$ has a center that has color $a$ in $c_V$. 
    Suppose that this is not the case. 
    This means that there is some vertex $y$ of color $b$ which is incident to at least two edges $xy, yz$ of color $a$, $a \neq b$.
    It follows from the construction of $c_E$ that each edge of color $a$ is incident to a vertex of color $a$.
    Hence $c_V(x) = c_V(z) = a$.
    This, however, means that $x$, $y$, and $z$ lie in the same component $H$ of $G_{a,b}$.
    Suppose that $x$ is closer to the root of $H$ than $z$. Then, however, $xy$ should be colored with $a$, and $yz$ should be colored with $b$, a contradiction.
    
    Note that, if there is an isolated edge of some color $a$ in $c_E$, then it is incident to exactly one vertex of color $a$, since $c_V$ is proper.
    With the previous observation that each nontrivial star of color $a$ has a center that has color $a$ in $c_V$, we get that the combination of $c_V$ and $c_E$ is indeed a TLIR coloring.
\end{proof}
Together with the result of Borodin
\begin{theorem}[Borodin~\cite{Borodin1979}]
Every planar graph has a vertex acyclic $5$-coloring.
\end{theorem}
we immediately get 
\begin{theorem}\label{thm_planar}
    Every planar graph $G$ has $\tlir(G) \leq 5$.
\end{theorem}

From the results on vertex acyclic colorings of graphs with maximum degree four~\cite{Burstein1979} and five~\cite{Kostochka2011}, we get:
\begin{corollary}
Let $G$ be a graph with maximum degree $\Delta$. If $\Delta=4$, then $\tlir(G) \leq 5$, and if $\Delta=5$, then $\tlir(G) \leq 7$.
\end{corollary}

Hakimi, Mitchem, and Schmeichel in~\cite{Hakimi1996} mentioned that maximal outerplanar graphs are a special subclass of 2-degenerate graphs.
By a maximal outerplanar graph we mean an outerplanar graph with the property that adding any edge between two nonadjacent vertices would result in a non-outerplanar graph.
It is straightforward to show that the minimum degree of each maximal outerplanar graph (on at least three vertices) is 2.

Let $v_n$ be a 2-vertex in a maximal outerplanar graph $G$.
Denote by $x$ and $y$ the neighbors of $v_n$.
Since $G$ is maximal outerplanar, $x$ and $y$ are adjacent, i.e., they form a clique.
$G - v_n$ is a smaller maximal outerplanar graph, hence, using an induction (the base case is $K_2$), we get
\begin{observation}\label{obs_outerplanar}
If $G$ is a maximal outerplanar graph, then there is an ordering $v_1, \dots, v_n$ of the vertices of $G$ such that each $v_i$ has at most $2$ neighbors with lower indices and these neighbors form a clique.
\end{observation}
For such a special class of 2-degenerate graphs (or $k$-degenerate in general) an acyclic coloring which uses at most $3$ (or $k+1$ colors in general) colors always exists:
\begin{lemma}[Hakimi, Mitchem, Schmeichel~\cite{Hakimi1996}]\label{lemma_special_degenerate}
If $G$ has an ordering of vertices $v_1, \dots, v_n$ of $G$ such that each $v_i$ has at most $k$ neighbors with lower indices and these neighbors form a clique, then $G$ has an acyclic $(k+1)$-coloring.
\end{lemma}
\begin{proof}
Color $G$ in a greedy way following the order of vertices $v_1, \dots, v_n$, i.e., when $v_i$ is colored, the smallest available color from $\{1, \dots , k+1\}$ is used.
Consider that after the coloring of some $v_i$, a 2-colored cycle was created. 
This would mean that $v_i$ has two neighbors, which are already colored (therefore they have smaller indices), and they received the same color.
Since we assume that the neighbors of $v_i$ with smaller indices form a clique, this is not possible.
Hence, such a greedy coloring is acyclic. 
\end{proof}

If $G$ is an outerplanar graph, we can add edges between nonadjacent vertices of $G$ to create a maximal outerplanar graph $G'$. 
Combining Observation~\ref{obs_outerplanar} and Lemma~\ref{lemma_special_degenerate} we get that $G'$ has an acyclic vertex 3-coloring. 
Since the removal of edges from $G'$ does not create cycles in any subgraph of $G'$ induced on vertices of two colors, this acyclic coloring of $G'$ is also an acyclic coloring of $G$. 
Hence, every outerplanar graph has an acyclic vertex 3-coloring, and using Theorem~\ref{thm_acyclic_to_TLIR} we get \begin{theorem}\label{thm_outerplanar}
    Every outerplanar graph $G$ has $\tlir(G) \leq 3$.
\end{theorem}

\end{document}